\documentclass[a4paper,12pt]{article}
\usepackage{amsmath}
\usepackage{amssymb}
\usepackage[utf8]{inputenc}
\usepackage[icelandic,english]{babel}
\usepackage{t1enc}
\usepackage{amsmath}
\usepackage{amssymb}
\usepackage{mathrsfs}
\usepackage[colorlinks=false,
pdfauthor={Benedikt Steinar Magnússon},
pdftitle={Envelopes of disc functionals and omega-plurisubharmonic functions}]{hyperref}

\renewcommand{\Bbb}{\mathbb}

\newcommand{\C}{{\Bbb  C}}
\newcommand{\D}{{\Bbb D}}

\newcommand{\R}{{\Bbb  R}}

\newcommand{\T}{{\Bbb  T}}
\newcommand{\A}{{\cal A}}

\newcommand{\F}{{\cal F}}

\renewcommand{\O}{{\cal O}}

\newcommand{\PSH}{{\operatorname{{\cal PSH}}}}
\newcommand{\SH}{{\operatorname{{\cal SH}}}}

\newcommand{\sing}{{\operatorname{sing}}}

\renewcommand{\phi}{\varphi}
\renewcommand{\epsilon}{\varepsilon}

\newcommand{\eq}{\begin{equation}}
\newcommand{\ee}{\end{equation}}

\def\squarebox#1{\hbox to #1{\hfill\vbox to #1{\vfill}}}

\newtheorem{theorem+}           {Theorem}      [section]
\newtheorem{definition+}  [theorem+]  {Definition}
\newtheorem{lemma+}  [theorem+]  {Lemma}
\newtheorem{corollary+}  [theorem+]  {Corollary}
\newtheorem{proposition+}  [theorem+]  {Proposition}
\newtheorem{example+}  [theorem+]  {Example}
\newtheorem{question+}  [theorem+]  {Question}

\newenvironment{theorem}{\begin{theorem+}\sl}{\end{theorem+}\rm}
\newenvironment{definition}{\begin{definition+}\rm}{\end{definition+}\rm}
\newenvironment{lemma}{\begin{lemma+}\sl}{\end{lemma+}\rm}
\newenvironment{corollary}{\begin{corollary+}\sl}{\end{corollary+}\rm}
\newenvironment{proposition}{\begin{proposition+}\sl}{\end{proposition+}\rm}

\newenvironment{proof}{\medbreak\noindent{\it Proof:}\rm}{\hfill$\square$\rm}
\newenvironment{prooftx}[1]{\medbreak\noindent{\it #1:}\rm}{\hfill$\square$\rm}

\title{\Large \bf Extremal $\omega$-plurisubharmonic functions as envelopes of
disc functionals - Generalization and applications to the local theory}
\author{\large\bf  Benedikt Steinar Magn\'usson }

\date{\today}

\begin{document}
\maketitle

\begin{abstract}
	We generalize the Poletsky disc envelope formula for
	the function $\sup \{u\in \PSH(X,\omega) \,;\, u\leq \varphi\}$ 
	on any complex manifold $X$
	to the case where the real $(1,1)$-current $\omega=\omega_1-\omega_2$ 
	is the difference of two positive closed $(1,1)$-currents and 
	$\varphi$ is the difference of an $\omega_1$-upper semicontinuous
	function and a plurisubharmonic function. 
\end{abstract}

\tableofcontents

\section{Introduction}\label{introdution}

Many of the extremal plurisubharmonic functions studied in pluripotential
theory are given as suprema of classes of plurisubharmonic functions
satisfying some bound which is given by a function $\varphi$.  
Some of these extremal functions can
be expressed as envelopes of disc functionals.
The purpose of this paper is to generalize a disc envelope formula  
for extremal $\omega$-plurisubharmonic functions of the form $\sup
\{u\in \PSH(X,\omega) \,;\, u\leq \varphi\}$ proved in 
\cite{Mag:2010}.  Our main result is the following

\begin{theorem}
	Let $X$ be a complex manifold, $\omega = \omega_1 - \omega_2$
	be the difference of two closed positive $(1,1)$-currents on
	$X$, $\phi=\varphi_1-\varphi_2$ be the difference
	of an  $\omega_1$-upper semicontinuous function $\varphi_1$
	in $L^1_{\text{loc}}(X)$ and a plurisubharmonic function $\varphi_2$,
	and assume that $\{ u \in \PSH(X,\omega) ; u \leq \phi \}$ is non-empty.
	Then the function $\sup \{ u \in \PSH(X,\omega) ; 
	u \leq \phi \}$ is $\omega$-plurisubharmonic and
	for every $x \in X \setminus \sing(\omega)$,
	\begin{multline*}
		\sup\{u(x) ; u \in \PSH(X,\omega), u \leq \phi \}\\
		= \inf\{ -R_{f^*\omega}(0) + \int_\T \phi \circ f\,
		d\sigma ;  f \in \A_X, f(0) = x \}.
	\end{multline*} \label{th_main}
	If $\{ u \in \PSH(X,\omega) ; u \leq \phi \}$ is empty, then the
	right hand side is $-\infty$ for every $x \in X$.  
	Here $\A_X$ denotes the set of all
	closed analytic discs in $X$, $\sigma$ is the arc length measure on
	the unit circle $\T$ normalized to $1$, and $R_{f^*\omega}$ is the
	Riesz potential in the unit disc $\D$ of the pull-back $f^*\omega$ 
	of the current $\omega$ by the analytic disc $f$. 
\end{theorem}

Observe that the supremum on the left hand side defines a function on 
$X$, but the infimum on the right hand side defines a function of $x$
only on $X \setminus \sing(\omega)$. The reason is that for 
$f \in \A_X$ with $f(0) = x \in \sing(\omega)$ both terms
$R_{f^*\omega}(0)$ and $\int_\T \phi \circ f\, d\sigma$ may take the
value $+\infty$ or the value $-\infty$ and in these cases it is 
impossible to define their difference in a sensible way.
The infimum is extended to $X$ by taking limes superior as explained
in Section 5.

The theorem generalizes a few well-known results.  Our main theorem 
in  \cite{Mag:2010} is the special case $\varphi_2=0$ and $\omega_2=0$.

The case $\varphi_2=0$ and $\omega=0$ is Poletsky's theorem,
originally proved by Poletsky \cite{Pol:1993} 
and Bu and Schachermayer \cite{BuSch:1992} for domains
$X$ in $\C^n$, and generalized to arbitrary manifolds by L\'arusson 
and Sigurdsson \cite{LarSig:1998,LarSig:2003} and Rosay \cite{Ros:2003}.
The case $\varphi_1=0$ and $\omega=0$ is a result of  
Edigarian \cite{Edi:2003}.  The case $\varphi_2=0$ and $\omega=0$ with
a weak notion of upper semi-continuity was also treated by
Edigarian \cite{Edi:2002}.  
The case when
$\varphi_1=\varphi_2=0$, $\omega_1=0$ and $\omega_2=dd^cv$, for a
plurisubharmonic function $v$ on $X$, was proved by L\'arusson 
and Sigurdsson in \cite{LarSig:1998,LarSig:2003}.

We combine the last case to the case when $\omega = 0$ in the following corollary,
which unifies the Poisson functional and the Riesz functional from \cite{LarSig:1998}. 
\begin{corollary}
	Assume $v$ is a plurisubharmonic function on a complex manifold $X$ and let 
	$\phi = \phi_1 - \phi_2$ be the difference of an upper semicontinuous function 
	$\phi_1$ and a plurisubharmonic function $\phi_2$. Then
	\begin{multline*}
		\sup\{ u(x) ; u \in \PSH(X), u \leq \phi, \mathscr L(u) \geq \mathscr L(v)\}\\
		= \inf\Big\{ \frac{1}{2\pi}\int_\D \log|\cdot|\, \Delta(v\circ f) 
		+ \int_\T \phi\circ f\, d\sigma ; f \in \A_X, f(0)=x \Big\}.
	\end{multline*}
\end{corollary}
Where $\mathscr L$ is the Levi form. 
This follows simply from the fact that if $\omega=-dd^c v$, then
$\PSH(X,\omega) = \{ u\in \PSH(X) ; \mathscr L(u) \geq \mathscr L(v) \}$ and
the Riesz potential $R_{f^*\omega}(0)$ is given by the first integral
on the right hand side.
Furthermore, since $\omega_1=0$ the function $\phi_1$ is $\omega_1$-usc
if and only if $\phi_1$ is usc.

The plan of the paper is the following.  In Section 2 we introduce the
necessary notions and results on $\omega$-upper semicontinuous
functions, $\omega$-plurisub\-harmonic functions, and analytic discs.  
In Section 3 we prove Theorem 1.1 in the special case 
when $\omega=0$.  In Section 4 we treat the case when the currents
$\omega_1$ and $\omega_2$ have global potentials. Section 5 contains an improved
version of the Reduction Theorem used in \cite{Mag:2010} which we use to
reduce the proof of Theorem 1.1 in the general case to the special case of global potentials.

This project was done under the supervision of my advisor Ragnar Sigurdsson, and I would like to thank him for his invaluable help writing this paper and
for all the interesting discussions relating to its topic.

\section{The $\omega$-plurisubharmonic setting}

First a few words about notation. We assume $X$ is a complex manifold of dimension $n$,
$\A_X$ will then be the \emph{closed analytic discs} in $X$, i.e.~the family
of all holomorphic mappings from a neighbourhood of the closed unit disc,
$\overline \D$, into $X$. The boundary of the unit disc $\D$ will be denoted by
$\T$ and $\sigma$ will be the arc length measure on $\T$ normalized to $1$. 
Furthermore, $D_r = \{ z\in \C ; |z| < r\}$ will be the disc centered at
zero with radius $r$.

We start by seeing that if $\omega$ is a closed, positive $(1,1)$-current on a manifold 
$X$, i.e.~acting on $(n-1,n-1)$-forms, then locally we have a potential for $\omega$, 
that is for every point $x$ there is a neighbourhood $U$ of $x$ and 
a psh function $\psi:U \to \R \cup \{-\infty\}$ such that $dd^c \psi = \omega$. This allows us 
to work with things locally in a similar fashion as the classical case, $\omega=0$. 
We will furthermore see that when there is a global potential, that is, when $\psi$ can be 
defined on all of $X$, then most of the questions about $\omega$-plurisubharmonic functions 
turn into questions involving plurisubharmonic functions.

Here we let $d$ and $d^c$ denote the real differential operators $d=\partial + \overline \partial$
and $d^c = i(\overline \partial - \partial)$. Hence, in $\C$ we have $dd^c u = \Delta u\, dV$
where $dV$ is the standard volume form.

\begin{proposition}
	 Let $X$ be a complex manifold with the second de Rham 
	 cohomology $H^2(X) = 0$, and the 
	 Dolbeault cohomology $H^{(0,1)}(X)=0$. Then every closed positive
	 $(1,1)$-current $\omega$ has a global
	 plurisubharmonic potential $\psi:X \to \R \cup \{-\infty\}$,
	 such that $dd^c \psi = \omega$.
	 \label{global_potential}
\end{proposition}
\begin{proof}
	Since $\omega$ is a positive current it is real, and from the fact $H^2(X) = 0$
	it follows that there is a real current $\eta$ such that $d\eta = \omega$.
	Now write
	$\eta = \eta^{1,0} + \eta^{0,1}$, where $\eta^{1,0} \in \Lambda'_{1,0}(X,\C)$
	and $\eta^{0,1} \in \Lambda'_{0,1}(X,\C)$.
	Note that $\eta^{0,1} = \overline{ \eta^{1,0}}$ since $\eta$ is real.
	We see, by counting degrees, that
	$\overline \partial \eta^{0,1} = \omega^{0,2} = 0$. 
	Then since $H^{(0,1)}(X)=0$,
	there is a distribution $\mu$ on $X$ such that
	$\overline \partial \mu = \eta^{0,1}$.
	Hence
	$$
		\eta = \overline{ \overline\partial \mu} + \overline\partial \mu
		= \partial \overline \mu + \overline \partial \mu.
	$$
	If we set $\psi = (\mu - \overline \mu)/2i$, then
	$$
		\omega = d\eta = d(\partial \overline \mu + \overline \partial \mu)
		= (\partial + \overline \partial)(\partial \overline \mu 
		+ \overline \partial \mu)
		= \partial \overline \partial (\mu - \overline \mu) 
		= dd^c \psi.
	$$
	Finally, $\psi$ is a plurisubharmonic
	function since $\omega$ is positive.
\end{proof}

If we apply this locally to a coordinate system biholomorphic to a polydisc and use
the Poincar\'e lemma we get the following.

\begin{corollary}
	For a closed, positive $(1,1)$-current $\omega$ there is 
	locally a plurisubharmonic potential
	$\psi$ such that $dd^c \psi = \omega$.
\end{corollary}

Note that the difference of two potentials for $\omega$ is a pluriharmonic function, thus $C^\infty$. 
So the \emph{singular set} $\sing(\omega)$ of $\omega$ is well defined as the union of all $\psi^{-1}(\{-\infty\})$
for all local potentials $\psi$ of $\omega$.

In our previous article \cite{Mag:2010} on disc formulas 
for $\omega$-plurisubharmonic functions 
we assumed that $\omega$ was a positive current. Here 
we can use more general currents and in the following we assume 
$\omega = \omega_1 - \omega_2$, where $\omega_1$ and $\omega_2$
are closed, positive $(1,1)$-currents. 
We have plurisubharmonic local potentials
$\psi_1$ and $\psi_2$ for $\omega_1$ and $\omega_2$, respectively, 
and we write the potential for $\omega$ as
$$
\psi(x) = \left\{ \begin{array}{ll}  
	\psi_1(x) - \psi_2(x) & \text{if } x \notin \sing(\omega_1) \cap \sing(\omega_2)\\
	\limsup\limits_{y\to x} \psi_1(y)-\psi_2(y) & \text{if } x \in \sing(\omega_1) \cap \sing(\omega_2)
\end{array}\right.
$$ 
and the singular set of $\omega$ is defined as
$\sing(\omega) = \sing(\omega_1) \cup \sing(\omega_2)$.

The reason for the restriction to $\omega = \omega_1 - \omega_2$, 
which is the difference of two positive, closed $(1,1)$-currents, 
is the following. Our methods rely on the existence of local potentials
which are well defined psh functions, not only distributions, for 
we need to apply 
Riesz representation theorem to this potential composed with an
analytic disc.
With $\omega = \omega_1 - \omega_2$ we can work with the 
local potentials of $\omega_1$ and $\omega_2$ separately, and they are 
are given by psh functions.

\begin{definition}
	A function $u:X \to [-\infty,+\infty]$ is called \emph{$\omega$-upper semicontinuous} ($\omega$-usc)
	if for every $a\in \sing(\omega)$, $\limsup_{X\setminus \sing(\omega)\ni z \to a} u(z) = u(a)$ and
	for each local potential $\psi$ of $\omega$, defined on an open subset $U$ of $X$, $u+\psi$
	is upper semicontinuous on $U \setminus \sing(\omega)$ and locally bounded above around each 
	point of $\sing(\omega)$.
	\label{def_usc}
\end{definition}

Equivalently, we could say that $\limsup_{\sing(\omega) \not\ni z\to a} u(z) = u(a)$ 
for every $a \in \sing(\omega)$ and
$u+\psi$ extends as
$$
	\limsup_{\sing(\omega) \not\ni z\to a} (u+\psi)(z), \qquad \text{for } a \in \sing(\omega)
$$ 
to an upper semicontinuous function on $U$ with values in 
$\R \cup \{-\infty\}$. This extension will be denoted $(u+\psi)^\dagger$.
Note that $(u+\psi)^\dagger$ is not the upper semicontinuous regularization $(u+\psi)^*$
of the function $u+\psi$, but just a way to define the sum on $\sing(\omega)$
where possibly one of the terms is equal to $+\infty$ and the other might
be $-\infty$.

\begin{definition}
	An $\omega$-usc function $u:X \to [-\infty,+\infty]$ is called 
	\emph{$\omega$-plurisub\-harmonic} ($\omega$-psh)
	if $(u+\psi)^\dagger$ is psh on $U$ for every local 
	potential $\psi$ of $\omega$ defined 
	on an open subset $U$ of $X$. We let $\PSH(X,\omega)$ denote the set of all
	$\omega$-psh functions on $X$.
	\label{def_psh}
\end{definition}

Similarly we could say that $u$ is $\omega$-psh if it is $\omega$-usc
and $dd^c u \geq \omega$.

As noted after Definition 2.1 in \cite{Mag:2010} the conditions on the 
values of $u$ at 
$\sing(\omega)$ are to ensure that $u$ is Borel measurable
and that $u$ is uniquely determined from its values outside of 
$\sing(\omega)$. 

If $\omega'$ and $\omega$ are cohomologous then the classes $\PSH(X,\omega')$
and $\PSH(X,\omega)$ are essentially translations of each other.

\begin{proposition}\label{prop_biject}
	Assume both $\omega$ and $\omega'$ are the difference of two positive, closed $(1,1)$-currents. 
	If the current $\omega - \omega'$ has a global potential 
	$\chi = \chi_1 - \chi_2 : X \to [-\infty, +\infty]$, 
	where $\chi_1$ and $\chi_2$ are psh functions,
	then for every $u'\in \PSH(X,\omega')$ the function $u$ defined
	by $u(x) = u'(x) - \chi(x)$ for $x \notin \sing(\omega')\cup
	\sing(\omega)$ extends to an unique function in $\PSH(X,\omega)$ 
	and the map $\PSH(X,\omega') \to \PSH(X,\omega)$, 
	$u'\mapsto u$ is bijective.
\end{proposition}
\begin{proof}
	Let $\psi' = \psi_1' - \psi_2'$ be a local potential of $\omega'$. The functions 
	$\psi_1 = \psi_1' + \chi_1$ and $\psi_2 = \psi_2' + \chi_2$ are well defined as the sums 
	of psh functions.
	Then $\psi = \psi_1 - \psi_2$, extended over $\sing(\omega)$ as before, 
	is a local potential of $\omega$ since $\omega = \omega' + dd^c \chi$.

	Take $u' \in \PSH(X,\omega')$ and define a function $u$ on $X$ 
	by
	$$u(x) = \left\{ \begin{array}{ll}
		(u'+\psi')^\dagger(x) - \psi(x) & \text{for } x \in X \setminus \sing(\omega)\\
		\limsup\limits_{\sing(\omega)\not\ni y \to x} (u'+\psi')^\dagger(y) - \psi(y) 
		& \text{for } x \in \sing(\omega) 
		\end{array}\right.
	$$
	This definition is independent of $\psi'$ since any other local potential of $\omega'$
	differs from $\psi'$ by a continuous pluriharmonic function which 
	cancels out in the definition of $u$, due to the definition of $\psi$.

	Then $u+\psi = (u'+\psi')^\dagger$ on $X \setminus \sing(\omega)$ where the sum is well defined,
	since neither $u$ nor $\psi$ are $+\infty$ there. The right hand side is usc so $u+\psi$ is
	usc on $X\setminus \sing(\omega)$. But $(u'+\psi')^\dagger$ is usc on $X$ so the extension
	$(u+\psi)^\dagger$ also satisfies $(u+\psi)^\dagger = (u'+\psi')^\dagger$ and is therefore 
	psh since $u' \in \PSH(X,\omega')$. This shows that $u \in \PSH(X,\omega)$. 

	This map from $\PSH(X,\omega')$ to $\PSH(X,\omega)$ is injective because 
	$u = u' - \chi$ almost everywhere and the extension over 
	$\sing(\omega) \cup \sing(\omega')$ is unique.
	
	By changing the roles of
	$\omega$ and $\omega'$ we get an injection in the opposite direction
	which maps $v \in \PSH(X,\omega)$ to a function 
	$v' \in \PSH(X,\omega)$ defined as
	$v'=v+\chi$ outside of 
	$\sing(\omega) \cup \sing(\omega')$.
	These maps are clearly the inverses of each other because if we
	apply the composition of them to the function $u' \in \PSH(X,\omega')$ 
	we get an $\omega$-upper semicontinuous function which satisfies
	$(u' -\chi) + \chi = u'$ outside of 
	$\sing(\omega) \cup \sing(\omega')$. Since this function is equal
	to $u'$ almost everywhere they are the same, 
	which shows that the composition
	is the identity map.	
\end{proof}

\begin{proposition}
	If $\phi\colon X \to [-\infty,+\infty]$ is an $\omega$-usc function
	we define $\mathcal F_{\omega,\phi} = \{ u \in \PSH(X,\omega) ; u \leq \phi \}$.
	If $\mathcal F_{\omega,\phi} \neq \emptyset$ then
	$\sup \mathcal F_{\omega,\phi} \in \PSH(X,\omega)$,
	and consequently $\sup \mathcal F_{\omega,\phi} \in \mathcal F_{\omega,\phi}$.
\end{proposition}
\begin{proof}
	Assume $\psi$ is a local potential of $\omega$ defined on $U\subset X$. 
	For $u \in \mathcal F_{\omega,\phi}$, the function $(u+\psi)^\dagger$ is a psh function on $U$ such that
	$(u+\psi)^\dagger \leq (\phi + \psi)^\dagger$. 
	The supremum of the family $\{ (u+\psi)^\dagger ; u \in \mathcal F_{\omega,\psi} \} \subset \PSH(U)$
	therefore defines a psh function $F_\psi(x) = ( \sup\{ (u+\psi)^\dagger(x) ; u \in \mathcal F_{\omega,\phi} \} )^*$ on $U$,
	with $F_\psi \leq (\phi +\psi)^\dagger$. We want to emphazise the difference
	between $\dagger$ and $*$. The extension of the function $u+\psi$ over
	$\sing(\omega)$, where the sum is possibly not defined, is denoted by
	$(u+\psi)^\dagger$ but $*$ is used to denote the upper semicontinuous 
	regularization of a function.

	Since the difference of two local potentials is a 
	continuous function, the function  
	$( \sup\{ (u+\psi)^\dagger ; u \in \mathcal F_{\omega,\phi} \} )^*-\psi$
	is independent of $\psi$. This means that
	$$
		S = F_\psi -\psi, \quad \text{on } U\setminus \sing(\omega),
	$$ 
	extended over $\sing(\omega)$ using $\limsup$, is a well-defined function on $X$. 
	
	Clearly 
	$S$ is $\omega$-psh since $(S+\psi)^\dagger = F_\psi$ which is psh, and $S$ satisfies
	$$
	\sup \mathcal F_{\omega,\phi} + \psi \leq F_\psi = S + \psi \leq \phi + \psi, \quad \text{on } U\setminus \sing(\omega).
	$$
	This implies 
	\begin{equation}
	\sup \mathcal F_{\omega,\phi} \leq S \leq \phi,
	\label{ineq_supF}
	\end{equation}
	on $U\setminus \sing(\omega)$. The later inequality holds also on $\sing(\omega)$ because of the definition of $S$
	at $\sing(\omega)$ and the $\omega$-upper semicontinuity of $\phi$.
	
	Furthermore, if $u\in \mathcal F_{\omega,\phi}$ and $a\in \sing(\omega)$, then
	\begin{equation*}
		u(a) = \limsup_{x\to a} u(x) \leq \limsup_{x\to a} [\sup \mathcal F_{\omega,\phi}(x)] \leq \limsup_{x\to a} S(x) = S(a).
	\end{equation*}
	Taking supremum over $u$ then shows that the first inequality in (\ref{ineq_supF}) holds also on $\sing(\omega)$. 
	Hence, $\sup \mathcal F_{\omega,\phi} \leq S$ and $S\in \mathcal F_{\omega,\phi}$, that is 
	$\sup \mathcal F_{\omega,\phi} = S \in \PSH(X,\omega)$.
\end{proof}

\begin{proposition}
	If $u,v \in \PSH(X,\omega)$ then $\max\{u,v\} \in \PSH(X,\omega)$.
\end{proposition}
\begin{proof}
	For any local potential $\psi$ we know that
	$\max\{u,v\} + \psi = \max\{u+\psi,v+\psi\}$ is usc outside of $\sing(\omega)$ 
	and locally bounded above around each point of $\sing(\omega)$. Therefore, the
	extension $(\max\{u,v\} + \psi)^\dagger$ is equal to 
	$\max\{(u+\psi)^\dagger,(v+\psi)^\dagger\}$ which is psh, hence $\max\{u,v\}$ is 
	$\omega$-psh.
\end{proof}

It is important for us to be able to define the pullback of $\omega$ by a 
holomorphic disc because it is needed to include $\omega$ in the disc functional 
for the case of $\omega$-psh functions in Chapters 4 and 5.

Assume $f(0) \notin \sing(\omega)$ and let $\psi$ be a local potential of
$\omega$. We define $f^*\omega$, the pullback of $\omega$ by $f$, locally by $dd^c (\psi \circ f)$. Since the difference of 
two local potentials is pluriharmonic, this definition is 
independent of the choice of $\psi$, and it 
gives a definition of $f^*\omega$ on all of $\D$.
Note that $\psi \circ f$ is not identically $\pm \infty$ since $f(0) \notin \sing(\omega)$.

If $\omega=\omega_1-\omega_2$, then we could as well define the 
positive currents $f^*\omega_1$ and $f^*\omega_2$, using 
$\psi_1$ and $\psi_2$ respectively, and then define $f^*\omega = f^*\omega_1 - f^*\omega_2$. 
This gives the same result since $\psi \circ f = \psi_1 \circ f - \psi_2 \circ f$ almost 
everywhere.

\begin{proposition}\label{equi}
	The following are equivalent for a function $u$ on $X$.
	\begin{enumerate}
		\item[(i)]\ $u$ is in $\PSH(X,\omega)$.
		\item[(ii)]\ $u$ is $\omega$-usc and $f^*u \in \SH(\D,f^*\omega)$ 
			for all $f\in \A_X$ such that $f(\D) \not\subset \sing(\omega)$.
	\end{enumerate}
\end{proposition}
	The proof is the same as the proof of Proposition 2.3 in \cite{Mag:2010}, where $\omega_2=0$.

\section{Proof in the case $\omega=0$}

We start by proving the main theorem in the case when $\omega_1 = \omega_2 = 0$. 
Note that if $\omega_1=0$ then $\omega_1$-upper semicontinuity is 
equivalent to upper semicontinuity.

In the following we assume $\phi_1$ is an usc
$L^1_\text{loc}$ function and $\phi_2$ is a psh function on a
complex manifold $X$. 
We define the function $\phi:X \to [-\infty,+\infty]$ by 
$$
		\phi(x) = \left\{ \begin{array}{ll}
		\phi_1(x) - \phi_2(x) & \text{if } \phi_2(x) \neq -\infty \\
		\limsup\limits_{\phi_2^{-1}(-\infty) \not\ni y \to x }\phi_1(y)-\phi_2(y) & \text{if } \phi_2(x) = -\infty.
		\end{array}\right.
	\label{eq_phi}
$$

Define $\A_X$ as the set of all closed analytic discs in $X$, that is holomorphic functions from a 
neighbourhood of the closed unit disc in $\C$ into $X$. 
The \emph{Poisson disc functional} $H_\phi:\A_X \to [-\infty,+\infty]$ 
of $\phi$ is defined as 
$H_\phi(f) = \int_\T \phi \circ f\, d\sigma$ for $f \in \A_X$, and
the \emph{envelope} $EH_\phi:X \to [-\infty,+\infty]$ of $H_\phi$ is defined as 
$$
	EH_\phi(x) = \inf\{ H_\phi(f); f\in \A_X, f(0) = x \}.
$$

The definition of the function $\phi$ should be viewed alongside
Lemma \ref{lemma_gooddiscs}, which states roughly that it  
suffices to look at discs not lying entirely in $\phi^{-1}(\{-\infty\})$.

Note that $\phi$ is a $L^1_\text{loc}$ function and that the Poisson functional
satisfies $H_\phi = H_{\phi_1} - H_{\phi_2}$, when $H_{\phi_1}(f) \neq -\infty$ or
$H_{\phi_2}(f) \neq -\infty$.

We start by showing that Theorem \ref{th_main} holds true on an open subset $X$ of $\C^n$ using
convolution.

Let $\rho:\C^n \to \R$  be a non-negative $C^\infty$ radial function with support 
in the unit ball $\mathbb B$ in $\C^n$ 
such that $\int_{\mathbb B} \rho\, d\lambda = 1$, where $\lambda$ is the
Lebesgue measure in $\C^n$. For an open set $X\subset \C^n$ we let
$X_\delta = \{ x \in X ; d(x,X^c) > \delta \}$ and if $\chi$ is in $L^1_\text{loc}(X)$
we define the convolution
$\chi_\delta(x) = \int_\mathbb B \chi(x - \delta y)\rho(y)\, d\lambda(y)$ which
is a $C^\infty$ function on $X_\delta$. It is well known that if $\chi \in \PSH(X)$
then $\chi_\delta \geq \chi$ and $\chi_\delta \searrow \chi$ as $\delta \searrow 0$.

\begin{lemma}
	Assume $X \subset \C^n$ is open and $\phi = \phi_1 - \phi_2$ as above. If
	$f \in \A_{X_\delta}$, then there exists $g \in \A_X$ such that $f(0)=g(0)$ and
	$H_\phi (g) \leq H_{\phi_\delta} (f)$, and consequently, $EH_\phi|_{X_\delta} \leq EH_{\phi_\delta}$.
	\label{H_ineq}
\end{lemma}
\begin{proof}
	Since $\phi_1$ is usc and $\phi_2$ is psh the function
	$(t,y) \mapsto \phi(f(t)-\delta y)$ is integrable on $\T \times \mathbb B$.
	By using the change of variables $y \to ty$ where $t\in \T$ and
	that $\rho$ is radial we see that
	\begin{eqnarray*}
		H_{\phi_\delta}(f)
		&=& \int_\T \int_{\mathbb B} 
		\phi (f(t)-\delta y)\rho(y)\, d\lambda(y)\, d\sigma(t)\\
		&=& \int_\T \int_{\mathbb B} 
		\phi (f(t)-\delta ty)\rho(y)\, d\lambda(y)\, d\sigma(t)\\
		&=& \int_{\mathbb B} \bigg( \int_\T 
		\phi (f(t)-\delta ty)\, d\sigma(t) \bigg) \rho(y)\, d\lambda(y).
	\end{eqnarray*}
	From measure theory we know that for every measurable function we can find a 
	point where the
	function is less than or equal to its integral with respect to a
	probability measure. Applying this to the function
	$y \mapsto \int_\T \phi (f(t)-\delta ty)\, d\sigma(t)$ 
	and the measure $\rho\, d\lambda$
	we can find $y_0 \in \mathbb B$ such that
	$$
		H_{\phi_\delta}(f) 
		\geq \int_\T \phi (f(t)-\delta ty_0)\, d\sigma(t) = H_{\phi}(g),
	$$
	if $g\in \A_X$ is defined by $g(t) = f(t)-\delta ty_0$. 
	It is clear that $g(0)=f(0)$.

	By taking the infimum over $f$, we see that $EH_\phi|_{X_\delta} \leq  EH_{\phi_\delta}$.  
\end{proof}

Note that
$EH_\phi|_{X_\delta}$ is the restriction of the function $EH_\phi$ to $X_\delta$, 
but not the envelope of the functional $H_\phi$ restricted to $\A_{X_\delta}$. 
There is a subtle difference between these two, and in general they are different.
The function $EH_{\phi_\delta}$ however, is only defined on $X_\delta$ since
the disc functional $H_{\phi_\delta}$ is defined on $\A_{X_\delta}$.

\begin{lemma}
	If $\phi = \phi_1 - \phi_2$ as above, then for every $f\in \A_X$ there is a limit
	$\lim_{\delta \to 0} H_{\phi_\delta}(f) \leq H_\phi(f)$ and it follows that for every
	$x \in X$, 
	$$
	\lim_{\delta \to 0} EH_{\phi_\delta} (x) = EH_\phi(x).
	$$
	\label{EH_lim}
\end{lemma}
\begin{proof}
	Let $f \in \A_X$, $\beta > H_\phi(f)$, and
	$\delta_0$ be such that $f(\overline \D) \in X_{\delta_0}$, and
	assume $\phi_2 \circ f \neq -\infty$. Since $\phi_2$ is plurisubharmonic we know
	that $\phi_{2,\delta} \geq \phi_2$ on $X_\delta$ for all $\delta < \delta_0$, so
	$$
		H_{\phi_\delta}(f) = 
		H_{\phi_{1,\delta}}(f) - H_{\phi_{2,\delta}}(f) \leq 
		\int_\T \sup_{B(f(t),\delta)} \phi_1\, d\sigma(t) - H_{\phi_2}(f).
	$$
	The upper semicontinuity of $\phi_1$ implies that the integrand on the
	right side is bounded above on $\T$ and also that it decreases to
	$\phi_1(f(t))$ when $\delta \to 0$.
	It follows from monotone convergence that the integral tends to
	$\int_\T \phi_1 \circ f\, d\sigma = H_{\phi_1}(f)$ when $\delta \to 0$, that is the right 
	side tends to $H_\phi(f) < \beta$.
	We can therefore find
	$\delta_1 \leq \delta_0$ such that  
	$$
		\int_\T \sup_{B(f(t),\delta)} \phi_1 \circ f\, d\sigma - H_{\phi_2}(f)< \beta,
	\quad\text{ for every } \delta < \delta_1.
	$$
	
	However, if $\phi_2 \circ f = -\infty$, then by monotone convergence
	\begin{eqnarray*}
		H_{\phi_\delta}(f) &=& \int_\T \int_\mathbb B \phi(f(t) - \delta y) \rho(y)\, d\lambda(y)\, d\sigma(t)\\
		&\leq& \int_\T \sup_{B(f(t),\delta)} \phi\, d\sigma(t) 
		= \int_\T \sup_{B(f(t),\delta)\setminus \phi_2^{-1}(-\infty)} (\phi_1 - \phi_2)\, d\sigma(t) \\
		&\xrightarrow[\delta \to 0]{}& \int_\T \limsup_{y \to f(t)} \big(\phi_1(y) - \phi_2(y)\big)\, d\sigma(t) 
		= H_\phi(f).
	\end{eqnarray*}

	This 
	along with the fact that
	$EH_\phi(x) \leq  EH_{\phi_\delta}(x)$ by Lemma \ref{H_ineq} 
	shows that
	$\lim_{\delta \to 0} EH_{\phi_\delta} = EH_{\phi}$.
\end{proof}

\begin{lemma}
	If $\phi=\phi_1-\phi_2$ as before, $f \in \A_X$, $f(\D) \subset 
	\phi_2^{-1}(-\infty)$, and $\epsilon >0$, 
	then there is a disc $g\in \A_X$ such that $g(\D) \not\subset \phi_2^{-1}(-\infty)$
	and $H_\phi(g) < H_\phi(f) + \epsilon$.
	\label{lemma_gooddiscs}
\end{lemma}
\begin{proof}
	By Lemma \ref{EH_lim} we can find $\delta >0$ such that $H_{\phi_\delta}(f) \leq H_\phi(f) + \epsilon$.
	Let $\tilde B = \big\{ y \in \mathbb B ;\{\phi(f(t)-\delta t y) ; t \in \D \} 
	\not\subset \phi_2^{-1}(-\infty) \big\}$, then $\mathbb B \setminus \tilde B$ is a zero set 
	and as before there is $y_0 \in \tilde B$ such that
	$$
		\int_\T \phi(f(t) - \delta t y_0)\, d\sigma(t)  
		\leq \int_\T \int_{\tilde B} \phi(f(t) - \delta t y) \rho(y)\, d\lambda(y)\, d\sigma(t)
		=  H_{\phi_\delta}(f).
	$$
	We define  $g \in \A_X$ by $g(t) = f(t) - \delta t y_0$. Then 
	$H_\phi(g) \leq H_\phi(f) + \epsilon$.
\end{proof}

\begin{lemma}
	Let $\phi$ be usc on a complex manifold $X$ and $F \in \O(D_r \times Y, X)$, where
	$r>1$ and $Y$ is a complex manifold, then $y \mapsto H_\phi(F(\cdot,y))$ is usc. 
	Furthermore, if $\phi$ is psh then this function is also psh.
	\label{lemma_composition}
\end{lemma}
\begin{proof}
	Fix a point $x_0 \in Y$ and a compact neigbourhood $V$ of $x_0$. 
	The function $\phi \circ F$ is usc and therefore bounded above on 
	$\T \times V$
	so by Fatou's lemma
	\begin{eqnarray*}
		\limsup_{x \to x_0} H_\phi(F(\cdot,x)) &\leq& \int_\T \limsup_{x\to x_0} \phi(F(t,x))\, d\sigma(t)\\
		= \int_\T \phi(F(t,x_0))\, d\sigma(t) &=& H_\phi(F(\cdot,x_0)),
	\end{eqnarray*}
	which shows that the function is usc.

	Assume $\phi$ is psh and let $h\in A_Y$. Then
	\begin{eqnarray*}
		\int_\T H_\phi(F(\cdot,h(s)))\, d\sigma(s) &=& \int_\T \int_\T \phi(F(t,h(s)))\,d\sigma(t)\,d\sigma(s)\\
		&=& \int_\T \int_\T \phi(F(t,h(s)))\,d\sigma(s)\,d\sigma(t) \\ 
		&\geq& \int_\T \phi(F(t,h(0)))\,d\sigma(t)\\
		&=& H_\phi(F(\cdot,h(0)),
	\end{eqnarray*}
	because for fixed $t$, the function $s\mapsto \phi(F(t,h(s)))$ is subharmonic.
\end{proof}

\begin{prooftx}{Proof of Theorem \ref{th_main} for an open subset $X$ of $\C^n$ and $\omega=0$}
	We start by showing that the envelope is usc. 
	
	Since $\phi_\delta$ is continuous we have by Poletsky's result 
	\cite{Pol:1993} that $EH_{\phi_\delta}$ is psh, in particular it is
	usc and does not take the value $+\infty$.

       Now, assume $x \in X$ and let $\delta >0$ be so small that $x \in X_\delta$.
       By the fact that $EH_{\phi_\delta}<+\infty$ and
       $EH_\phi|_{X_\delta} \leq EH_{\phi_\delta}$
       we see that $EH_\phi$ is finite.

       For every $\beta > EH_\phi(x)$, we let $\delta>0$ be such that
       $EH_{\phi_\delta}(x) < \beta$. Since $EH_{\phi_\delta}$ is upper
       semicontinuous there is a neighbourhood $V \subset X_\delta$ of $x$
       where $EH_{\phi_\delta} < \beta$.
       By Lemma \ref{H_ineq} we see that
       $EH_\phi < \beta$ on $V$, which
       shows that $EH_\phi$ is upper semicontinuous.

	Now we only have to show that $EH_\phi$ 
	satisfies the sub-average property.

	Fix a point $x \in X$, an analytic disc $h\in \A_X$, $h(0)=x$ 
	and find $\delta_0$ such that
	$h(\overline \D) \subset X_{\delta_0}$.
	Note that the function $EH_{\phi_\delta}$ is psh by Poletsky's result \cite{Pol:1993} since $\phi_\delta$ is
	continuous. 
	Then Lemma \ref{H_ineq} and the plurisubharmonicity of $EH_{\phi_\delta}$ 
	gives that for every $\delta < \delta_0$, 
	$$
		EH_\phi(x) \leq EH_{\phi_\delta}(x) \leq 
		\int_\T EH_{\phi_\delta}\circ h\, d\sigma.
	$$
	When $\delta \to 0$ Lebesgue's theorem along with Lemma \ref{EH_lim} implies that
	$EH_\phi(x) \leq \int_\T EH_\phi \circ h\, d\sigma$.

	Since $EH_\phi(x) \leq H_\phi(x) = \phi(x)$, where $H_\phi(x)$ is the functional $H_\phi$ evaluated at the constant
	disc $t \mapsto x$, we see that $EH_\phi \leq \sup \F_\phi$.

	Also, if $u \in \F_\phi$
	and $f \in \A_X$, then
	$$
		u(f(0)) \leq \int_\T u \circ f\, d\sigma \leq \int_\T \phi \circ f\, d\sigma = H_\phi(f).
	$$
	Taking supremum over $u \in \F_\phi$ and infimum over $f \in \A_X$ we get the opposite inequality,
	$\sup \mathcal F_\phi \leq EH_\phi$, and therefore an equality.
\end{prooftx}

For the case when $X$ is a manifold we need the following theorem of L\'arusson and Sigurdsson
(Theorem 1.2 in \cite{LarSig:2003}). 

\begin{theorem}
	A disc functional $H$ on a complex manifold $X$ has a plurisubharmonic envelope 
	if it satisfies the following three conditions.
	\begin{itemize}
		\item[(i)] The envelope $E\Phi^*H$ is plurisubharmonic for every 
			holomorphic submersion $\Phi$ from
			a domain of holomorphy in affine space into $X$, where 
			the pull-back $\Phi^*H$ is defined as $\Phi^*H(f)=H(\Phi \circ f)$
			for a closed disc $f$ in the domain of $\Phi$.
		\item[(ii)] There is an open cover of $X$ by subsets $U$ with a pluripolar 
			subset $Z \subset U$ such that
			for every $h \in \A_U$ with $h(\overline \D) \not\subset Z$, the 
			function $w\mapsto H(h(w))$
			is dominated by an integrable function on $\T$.
		\item[(iii)] If $h \in \A_X$, $w\in\T$, and $\epsilon >0$, then $w$ has a 
			neighbourhood $U$ in $\C$ such that
			for every sufficiently small closed arc $J$ in $\T$ 
			containing $w$ there is a holomorphic map
			$F:D_r \times U \to X$, $r>1$, such that $F(0,\cdot) = h|_U$ and 
			\begin{equation}
				\frac{1}{\sigma(J)} \underline{\int_J} H(F(\cdot,t))\, d\sigma(t) 
				\leq EH(h(w)) + \epsilon,
				\label{}
			\end{equation}
			where the integral on the left hand side is the lower integral, i.e.~the 
			supremum of the integrals of all integrable Borel functions dominated by the integrand.
	\end{itemize}
	\label{th_larsig}
\end{theorem}

\begin{prooftx}{Proof of Theorem \ref{th_main} for a general complex manifold $X$ and $\omega=0$}
	We have to show that $H_\phi$ satisfies the three condition in Theorem \ref{th_larsig}. 
	Condition \textsl{(i)} follows from 
	the case above when $X \subset \C^n$ and condition \textsl{(ii)} if we take $U=X$ and $Z=\phi^{-1}(\{+\infty\})$.
	Then $H_\phi(h(w)) = \phi(h(w))$ which is integrable since $h(0) \notin Z$.
	
	To verify condition \textsl{(iii)}, let $h\in \A_X$, $w\in \T$ and $\beta > EH_\phi(h(w))$. 
	Then there is a disc $f \in \A_X$, $f(0)=h(w)$ such that $H_\phi(f) < \beta$. 
	Now look at the graph $\{ (t,f(t) )  \}$ of $f$ in $\C \times X$ and let $\pi$ denote the projection 
	from $\C \times X$ to $X$.
	As in the proof of Lemma 2.3 in \cite{LarSig:1998} 
	there is, by restricting the graph to a disc $D_r$, $r>1$, a bijection $\Phi$ from a neighbourhood of the 
	graph onto $\D^{n+1}$ such that $\Phi(t,f(t)) = (t,\overline 0)$. In order to clarify the notation we write
	$\overline 0$ for the zero vector in $\C^n$.
	
	If we define $\tilde \phi = \phi \circ \pi \circ \Phi^{-1}$, 
	then $H_\phi(f) = H_{\tilde \phi}((\cdot,\overline 0))$, where 
	$(\cdot,\overline 0)$ represents the analytic disc
	$t \mapsto (t,0,\ldots,0)$. The function $\tilde \phi$ is defined on an open subset of $\C^{n+1}$
	which enables us to smooth it using convolution as in the first part of this section. 

	By Lemma \ref{EH_lim}, there is $\delta \in ]0,1[$ such that 
	$H_{\tilde \phi_\delta}( (\cdot,\overline 0 ) ) < \beta$.
	Since $\tilde \phi_\delta$ is continuous, the function 
	$x \mapsto H_{\tilde \phi_\delta}( (\cdot,\overline 0) + x)$ is continuous.
	Then there is a neighbourhood $\tilde U$ of
	$0$ in $D_{1-\delta}^n$, such that 
	$H_{\tilde \phi_\delta}( (\cdot,\overline 0) + x) < \beta$ for  $x \in \tilde U$. 
	Let $J \subset \T$ be a closed arc such that $\tilde h(J) \subset \tilde U$, where $\tilde h(t) = \Phi(0,h(t))$.

	With the same argument as in the proof of 
	Lemma \ref{H_ineq}, we find $y_0 \in \mathbb B \subset \C^{n+1}$ such that
	\begin{eqnarray*}
		\beta &>& \frac{1}{\sigma(J)} \int_J H_{\tilde \phi_\delta}\big( (\cdot,0) + \tilde h(t) \big)\, d\sigma(t)\\
		&=& \frac{1}{\sigma(J)}\int_{\mathbb B} \Big( \int_J \int_\T \tilde \phi \big( (s,0) + h(t) - \delta s y \big)
		\, d\sigma(s)\, d\sigma(t) \Big) \rho(y) \, d\lambda(y) \\
		&\geq& \frac{1}{\sigma(J)}\int_J \int_\T \tilde \phi \big( (s,0) + \tilde h(t) - \delta s y_0 \big)
		\, d\sigma(s)\,d\sigma(t).
	\end{eqnarray*}
	We define the function $F \in \O(D_{r} \times U, X)$ by 
	$$
		F(s,t) = \pi \circ \Phi^{-1}( (s,0) + \Phi(0,h(t)) - \delta s y_0 )
	$$ 
	and the set $U = h^{-1}(\pi(\Phi^{-1}(\tilde U)))$.
	
	Then $\tilde \phi ((s,0) + \tilde h(t) -\delta s y_0) = \phi(F(s,t)$, and we conclude that
	$$
	 \beta > \frac{1}{\sigma(J)}\int_J \int_\T \phi( F(s,t) )\, d\sigma(s)\, d\sigma(t)
		= \frac{1}{\sigma(J)}\int_J H_\phi( F(\cdot,t))\,d\sigma(t).
	$$

\end{prooftx}

\section{Proof in the case of a global potential}

We now look at the case when $\omega = \omega_1 - \omega_2$ has a global potential,
and show how Theorem 1.1 then follows from the results in Section 3.
We first assume $\phi_2 = 0$, that is the weight $\phi = \phi_1$ is an $\omega_1$-usc function. 

The Poisson disc functional from Section 3 is obviously not appropriate here since
it fails to take into account the current $\omega$. The remedy is to look at the pullback of
$\omega$ by an analytic disc. If $f$ is an analytic disc we define a closed
$(1,1)$-current $f^*\omega$ on $\D$ in exactly the same way as in \cite{Mag:2010}.

Assume $f(0) \notin \sing(\omega)$ and let $\psi$ be a local potential of
$\omega$. We define $f^*\omega$ locally by $dd^c (\psi \circ f)$. Because the difference of 
two local potentials is pluriharmonic then this is independent of the choice of $\psi$, so it 
gives a definition of $f^*\omega$ on all of $\D$.
Note that $\psi \circ f$ is not identically $\pm \infty$ since $f(0) \notin \sing(\omega)$.

We could as well define the positive currents $f^*\omega_1$ and $f^*\omega_2$, using 
$\psi_1$ and $\psi_2$ respectively, and then define $f^*\omega = f^*\omega_1 - f^*\omega_2$. 
This gives the same result since $\psi \circ f = \psi_1 \circ f - \psi_2 \circ f$ almost 
everywhere.

It is also possible to look at $f^*\omega$ as a real measure on $\D$, and as 
before, we let $R_{f^*\omega}$ be its Riesz potential,
\begin{equation}\label{R_f}
	R_{f^*\omega}(z) = \int_\D G_\D(z,\cdot)\, d(f^*\omega),
\end{equation}
where $G_\D$ is the Green function for the unit disc, $G_\D(z,w) = 
\frac{1}{2\pi}\log\frac{|z-w|}{|1-z\overline w|}$.
Since $f$ is a closed analytic disc not lying in $\sing(\omega)$ it follows
that $f^*\omega$ is a Radon measure 
in a neighbourhood of the unit disc, therefore with finite mass on $\D$ and not 
identically $\pm \infty$.

It is important to note that if  we have a local potential $\psi$ defined in a neighbourhood of 
$\overline{f(\D)}$, then the Riesz
representation formula, Theorem 3.3.6 in \cite{Hor}, at the point 0 gives
\begin{equation}\label{riesz}
	\psi(f(0)) = R_{f^*\omega}(0) + \int_\T \psi \circ f\, d\sigma.
\end{equation}
 
Next we define the disc functional. We let $\phi$ be an $\omega_1$-usc function on  
$X$ and fix a point $x \in X\setminus \sing(\omega)$. Let $f\in \A_X$, $f(0) = x$ and let 
$u \in \F_{\omega,\phi}$, where $\F_{\omega,\phi} = \{ u \in \PSH(X,\omega) ; u\leq \phi \}$.
By Proposition \ref{equi}, $u \circ f$ is
$f^*\omega$-subharmonic on $\D$, and since the Riesz potential $R_{f^*\omega}$ is a global potential
for $f^*\omega$ on $\D$ we have, by the subaverage property of $u\circ f + R_{f^*\omega}$, that
$$
	u(f(0)) + R_{f^*\omega}(0) \leq \int_\T u \circ f\, d\sigma + \int_\T R_{f^*\omega}\, d\sigma.
$$
Since, $R_{f^*\omega} = 0$ on $\T$ and $u\leq \phi$, we conclude that
$$
	u(x) \leq  - R_{f^*\omega}(0) + \int_\T \phi\circ f\, d\sigma.
$$

The right hand side is independent of $u$ so we can define the functional $H_{\omega,\phi}:\A_X \to [-\infty,+\infty]$
by
$$
H_{\omega,\phi}(f) = -R_{f^*\omega}(0) + \int_\T \phi \circ f\, d\sigma.
$$

By taking the supremum on the left hand side over all $u\in \PSH(X,\omega)$, $u \leq\phi$, and
the infimum on the right hand side over all $f \in \A_X$, $f(0)=x$ we get the inequality
\begin{equation}
	\sup \F_{\omega,\phi} \leq EH_{\omega,\phi}, \qquad \text{on } X\setminus \sing(\omega).
	\label{fund_ineq}
\end{equation}
We wish to show that this is an equality. 
By applying $H_{\omega,\phi}$ to the constant discs in $X\setminus \sing(\omega)$ 
we see that the right hand side 
is not greater than $\phi$. If we show that $EH_{\omega,\phi}$ is $\omega$-psh then it is
in $\F_{\omega,\phi}$ and we have an equality above. 

\begin{lemma}
	If $f \in \A_X$ and $\psi = \psi_1 - \psi_2$ is a potential for $\omega$ in a 
	neighbourhood of $f(\overline \D)$ then
	$$
		H_{\omega,\phi}(f) + \psi(f(0)) = H_{\phi+\psi}(f).
	$$
	\label{le_globalpotential}
\end{lemma}
\begin{proof}
	By the linearity of $R_{f^*\omega}$ and 
	Riesz representation (\ref{riesz}) of $f^*\psi_1$ and $f^*\psi_2$ we get
	\begin{eqnarray*}
		H_{\omega,\phi}(f) + \psi(f(0)) &=& 
		-R_{f^*\omega}(0) + \int_\T \phi \circ f\, d\sigma +\psi(f(0))\\
		&=& -R_{f^*\omega}(0) + \int_\T \phi \circ f\, d\sigma 
		+ R_{f^*\omega}(0) + \int_\T (\psi_1-\psi_2)\circ f\, d\sigma\\
		&=& \int_\T (\phi + \psi_1 - \psi_2) \circ f\, d\sigma = H_{\phi+\psi}(f).
	\end{eqnarray*}
\end{proof}

\begin{prooftx}{Proof of Theorem \ref{th_main} in the case when $\omega_1$ and $\omega_2$ have
 global potentials and $\phi_2=0$}
	By Lemma \ref{le_globalpotential} for $x \in X\setminus \sing(\omega)$,
	$$
		EH_{\omega,\phi}(x) + \psi(x) = \inf\{H_{\omega,\phi}(f) + \psi(x) ; f\in \A_X, f(0)=x\}
		= EH_{\phi +\psi}(x).
	$$
	Since $\phi + \psi = (\phi +\psi_1) - \psi_2$ is the difference of an usc 
	function and a plurisubharmonic function,
	the result from Section 3 gives that $EH_{\phi+\psi}$ is psh and 
	equivalently $EH_{\omega,\psi}$ is $\omega$-psh.
\end{prooftx}

\section{Reduction to global potentials and end of proof}

The purpose of this section is to generalize the reduction theorem presented in
\cite{Mag:2010} and simplify the proof of it. Then we apply it to the result in Section 4 to
finish the proof of Theorem \ref{th_main}.

The proof of the Reduction Theorem here does not directly rely on the construction of a Stein manifold 
in $\C^4 \times X$, instead we use Lemma \ref{lemma_potential} below to define a local potential 
around the graphs of the appropriate discs in $\C^2 \times X$.

It should be pointed out that Theorem \ref{th_red} does not work specifically with the Poisson functional but
a general disc functional $H$. We will however apply the results here to the Poisson functional
from Section 4, so it is of no harm to think of it in the role of $H$.

If $H$ is a disc functional defined for discs $f \in \A_X$, with 
$f(\D) \not\subset \sing(\omega)$, then 
we define the envelope $EH$ of $H$ on $X\setminus \sing(\omega)$ by 
$$
	EH(x) = \inf\{ H(f) ; f \in \A_X, f(0) = x \}.
$$
We then extend $EH$ to a function on $X$ by 
\begin{equation}
	EH(x) = \limsup_{\sing(\omega) \not\ni y \to x} EH(y), \qquad \text{for } x \in \sing(\omega), 
	\label{EH_sing}
\end{equation}
in accordance with Definition \ref{def_usc} of $\omega$-usc functions.

If $\Phi:Y \to X$ is a holomorphic function and $H$ a disc functional on $\A_X$, 
then we can define the pullback 
$\Phi^*H$ of $H$ by $\Phi^*H(f) = H(\Phi\circ f)$, for $f \in \A_Y$. 
Every disc $f \in \A_Y$ gives a push-forward $\Phi\circ f \in A_X$ and 
it is easy to see that
\begin{equation}
	\Phi^*EH \leq E\Phi^*H,
	\label{EPhiH}
 \end{equation}
where $\Phi^*EH = EH \circ \Phi$ is the pullback of $EH$. 
We have an equality in (\ref{EPhiH}) if every disc $f \in \A_X$ has a lifting 
$\tilde f \in \A_Y$, $f = \Phi \circ \tilde f$.

If $\Phi:Y \to X$ is a submersion the currents $\Phi^*\omega_1$ and $\Phi^*\omega_2$
are well-defined on $Y$. 
The core in showing the $\omega$-plurisubharmonicity of $EH$ is the following lemma. 
It produces 
a local potential of the currents $\Phi^*\omega_1$ and $\Phi^*\omega_2$ 
in a neighbourhood of the graphs of the discs from condition \textsl{(iii)} 
in Theorem \ref{th_red} below.
\begin{lemma}
	Let $X$ be a complex manifold and $\tilde \omega$ a positive closed $(1,1)$-current on $\C^2 \times X$.
	Assume $h \in \O(D_r,X)$, $r>1$ and for $j=1,\ldots,m$ assume $J_j \subset \T$ are 
	disjoint arcs and $U_j \subset D_r$ are pairwise disjoint open discs containing $J_j$.
	Furthermore, assume there are functions $F_j \in \O(D_s \times U_j,X)$, $s>1$, for $j=1,\ldots,m$, such that
	$F_j(0,w) = h(w)$, $w\in U_j$.

	If $K_0 = \{ (w,0,h(w)) ; w \in \overline \D \}$ and $K_j = \{ (w,z,F_j(z,w)) ; z \in \overline \D, w \in J_j \}$
	then there is an open neighbourhood 
	of $K = \cup_{j=0}^m K_j$ where $\tilde \omega$ has a global potential $\psi$.
	\label{lemma_potential}
\end{lemma}
\begin{proof}
	For convenience we let $U_0 = D_r$ and $F_0(z,w) = h(z)$, also
	$\overline 0$ will denote the zero vector in $\C^n$. 
	The graphs of the $F_j$'s are biholomorphic to polydiscs, hence Stein. 
	By slightly shrinking the $U_j$'s and $s$ we can, 
	just as in the proof of Theorem 1.2 in \cite{LarSig:2003}, 
	use Siu's Theorem \cite{Siu:1976} and the proof of Lemma 2.3 in \cite{LarSig:1998} 
	to define 
	biholomorphisms $\Phi_j$ from the polydisc $U_j \times D_s^{n+1}$ onto a neighbourhood of the $K_j$ such that
	\begin{equation}
		\Phi_j(w,z,\overline 0) = (w,z,F_j(z,w)), \qquad w\in U_j, z\in D_s,
		\label{Phi_j}
	\end{equation}
	for $j=1,\ldots,m$ and 
	\begin{equation}
		\Phi_0(w,0,\overline 0) = (w,0,h(w)), \qquad w \in U_0.
		\label{Phi_0}
	\end{equation}
	Furthermore, we may assume that the maps $\Phi_j$ are continuous on the closure of
	$U_j \times D_s^{n+1}$ for $j=0,\ldots,m$.

	For $j=1,\ldots,m$ let $U_j'$ and $U_j''$ be open discs concentric to $U_j$ such that
	$$
		J_j \subset\subset U_j'' \subset\subset U_j' \subset\subset U_j,
	$$
	and $B_j$ a neighbourhood of $\Phi_j(\overline{U_j'} 
	\times \{(0, \overline 0)\})$ defined by
	$$
		B_j = \Phi_j(U_j \times D_{\delta_j}^{n+1})
	$$
	for $\delta_j >0$ small enough so that
	$$
		B_j \subset \Phi_0(U_0 \times D_s^{n+1}),
	$$
	and
	$$
		B_j \cap  K_k = \emptyset, \text{when } k\neq j \text{ and } k\geq 1.
	$$
	This is possible since $\Phi_j(U_j \times \{(0,\overline 0) \}) 
	\subset \Phi_0(U_0\times D_s^{n+1})$
	and $\Phi_j(U_j\times \{(0,\overline 0)\}) \cap K_k = 
	\emptyset$ if $k\neq j$ and $k\geq 1$.

	The compact sets $\Phi_0(\overline{U_0} \setminus U_j' 
	\times \{(0,\overline 0) \})$ and 
	$\Phi_j(\overline{U_j''}\times \overline D_s \times \{ \overline 0 \} )$ are
	disjoint by (\ref{Phi_0}) and (\ref{Phi_j}), and likewise 
	$\Phi_0(\overline{U_j'} \times \{ (0,\overline 0)\} ) \subset\subset B_j$. 
	So there is a $\epsilon_j > 0$ such that
	$$
		\Phi_0(U_0\setminus U_j' \times D_{\epsilon_j}^{n+1}) \cap 
		\Phi_j(U_j''\times D_s \times D_{\epsilon_j}^n) = \emptyset
	$$
	and 
	$$
		\Phi_0(U_j' \times D_{\epsilon_j}^{n+1}) \subset B_j.
	$$
	
	Let $\epsilon_0 = \min\{\epsilon_1,\ldots,\epsilon_m\}$ and define $V_0 = \Phi_0(U_0 \times D_{\epsilon_0}^{n+1})$ and
	$V_j = \Phi_j(U_j''\times D_s \times D_{\epsilon_j}^n)$.

	Furthermore, since the graphs of the $F_j$'s, $\Phi_j(U_j\times D_s \times \{ \overline 0 \})$, are disjoint for $j\geq 1$ we may
	assume $V_j \cap V_k = \emptyset$, and similarly that $B_j \cap B_k = \emptyset$ when $j\neq k$ and $j,k\geq 1$.

	What this technical construction has achieved is to ensure the intersection $V_0 \cap V_j$ is contained in $B_j$, while
	still letting all the sets $V_j$ and $B_j$ be biholomorphic to polydiscs. Then both
	$V = \cup_{j=1}^m V_j$ and $B = \cup_{j=1}^m B_j$ are disjoint unions of polydiscs.
		
	By Proposition \ref{global_potential} there are local potentials $\psi_j$ of $\tilde\omega$ on each of the sets
	$\Phi_j(U_j \times D_s^{n+1})$, $j=1,\ldots,m$.

	Define $\eta' = d^c\psi_0$ on $V_0 \cup B$ and $\eta''$ on $V\cup B$ by 
	$\eta'' = d^c\psi_j$ on $V_j \cup B_j$, this is well defined because the $V_j \cup B_j$'s are
	pairwise disjoint and $V_j \cup B_j \subset \Phi_j(U_j \times D_s^{n+1})$.
	Since $d\eta' - d\eta'' = \tilde\omega-\tilde\omega=0$ on $B$ there is a distribution
	$\mu$ on $B$ satisfying $d\mu=\eta'-\eta''$.

	Let $\chi_1,\chi_2$ be a partition of unity subordinate to the covering
	$\{V_0,V\}$ of $V_0 \cup V$. Then
	$$\eta= \left\{ \begin{array}{ll}
		\eta' - d(\chi_1 \mu) & \text{on } V_0\\
		\eta'' + d(\chi_2 \mu) & \text{on } V
		\end{array}\right.
	$$
	is well defined on $V_0 \cup V$ with $d\eta = \tilde \omega$.
	
	If we repeat the topological construction above for $V_0,\ldots,V_m$ instead of $\Phi_j(U_j\times D_s^{n+1})$ we can
	define sets $V_0',\ldots,V_m'$ and $B_1',\ldots,B_m'$ biholomorphic to polydiscs such that $V_j' \subset V_j$, $B_j' \subset B_j$ and
	$$
		V_0' \cap V_j' \subset B_j' \subset V_0 \cap V_j,
	$$ 
	and both the $B_j'$'s and the $V_j'$'s are pairwise disjoint. Define
	$V' = \cup_{j=1}^m V'_j$.

	Let $\psi'$ be a real distribution defined on $V_0$ satisfying $d^c \psi' = \eta' -d\chi_1\mu$ and 
	let $\psi''$ be a real distribution defined on $V$ satisfying $d^c \psi'' = \eta'' -d\chi_2\mu$. 
	Then $d^c(\psi'-\psi'') = \eta' - \eta'' - d(\chi_1\mu + \chi_2\mu) = 0$.
	Therefore, on each of the connected sets $B_j'$ we have $\psi'-\psi'' = c_j$, for some constant $c_j$.
	Consequently the distribution $\psi$ is well defined on $V_0' \cup V'$ by
	$$\psi= \left\{ \begin{array}{ll}
		\psi' & \text{on } V_0'\\
		\psi + c_j & \text{on } V_j'
		\end{array}\right.
	$$
	since $V_0' \cap V' \subset B'$ and the $V_j'$'s are disjoint.
	It is clear that $dd^c \psi = d\eta = \tilde \omega$ and since $\omega$ is positive we may
	assume $\psi$ is a plurisubharmonic function.	
\end{proof}

We now turn our attention back to the $\omega$-plurisubharmonicity of the
envelope $EH$. We start by showing that it is $\omega$-usc, but 
this is done separately because it needs weaker assumptions than those needed 
in Theorem \ref{th_red} where we show that $EH$ is $\omega$-psh.

\begin{lemma}
	Let $X$ be an $n$-dimensional complex manifold, $H$ a disc functional 
	on $\A_X$, and $\omega = \omega_1 - \omega_2$ the 
	difference of two positive, closed $(1,1)$-currents on $X$.
	The envelope $EH$ is $\omega$-usc if $E\Phi^*H$ is $\Phi^*\omega$-usc
	for every submersion $\Phi$ from a set biholomorphic to a 
	$(n+1)$-dimensional polydisc into $X$.
	\label{lemma_redusc}
\end{lemma}
\begin{proof}
	To show that $EH+\psi$ does not take the value $+\infty$ at $x \in X \setminus \sing(\omega)$,
	let $U$ be a coordinate polydisc in $X$ centered at $x$ and $\psi$ a local potential of $\omega$
	on $U \subset X$. Then by (\ref{EPhiH}),
	\begin{equation*}
		EH(x)+\psi(x) = EH(\Phi(0,x))+\psi(\Phi(0,x))) \leq E\Phi^*H( (0,x)) + \psi(\Phi(0,x)) < +\infty,
	\end{equation*}
	where $\Phi:\D \times U \to U$ is the projection.

	Let $\beta > EH(x)$ and $g\in A_X$ such that $H(g) < \beta$. By a now familiar argument in Lemma
	2.3 in \cite{LarSig:1998} there is a biholomorphism $\Psi$ from a neighbourhood of the graph
	$\{ (w,g(w)) ; w \in \D \}$ into $D_s^{n+1}$, $s>1$ such that $\Psi(w,g(w)) = (w,\overline 0)$.
	If $\Phi$ is the projection $\C \times X \to X$ then $\Phi^*\psi = \psi \circ \Phi$ is a local potential 
	of $\Phi^*\omega$ on $\C \times U$. Now, if $\tilde g \in \A_{\C\times X}$ is the lifting
	$w \mapsto (w,g(w))$ of $g$ then by (\ref{EPhiH}),
	\begin{equation*}
		E\Phi^*H( (0,x) ) + \psi(\Phi( (0,x) ) \leq 
		\Phi^*H(\tilde g) + \psi(\Phi( (0,x) )) =
		H(g) + \psi(x) < \beta.
	\end{equation*} 
	By assumption there is a neighbourhood $W_0 \times W \subset \C \times U$ of 
	$(0,x)$ such that for $(z_0,z) \in W_0 \times W$,
	\begin{equation*}
		E\Phi^*H( (z_0,z) ) + \psi(\Phi( (z_0,z) )) < \beta.
	\end{equation*}
	Then by (\ref{EPhiH}), $EH(z)+\psi(z) \leq \beta$ for $z \in W$ which 
	shows that $EH+\psi$ is usc outside of $\sing(\omega)$ and by 
	(\ref{EH_sing}), the definition of 
	$EH$ at $\sing(\omega)$, we have shown that $EH$ is $\omega$-usc.
\end{proof}

The following theorem shows that an envelope $EH$ is $\omega$-psh if it satisfies some conditions which are 
almost identical to those in Theorem 4.5 in \cite{Mag:2010}. These conditions are
very similar to those posed upon the envelope in Theorem \ref{th_larsig} when $\omega=0$.

\begin{theorem}{(Reduction theorem):}
	Let $X$ be a complex manifold, $H$ a disc functional on $\A_X$ and 
	$\omega = \omega_1 - \omega_2$ the
	difference of two positive, closed $(1,1)$-currents on $X$.
	The envelope $EH$ is $\omega$-plurisubharmonic if it satisfies the following.
	\begin{enumerate}
		\item[(i)] $E\Phi^*H$ is $\Phi^*\omega$-plurisubharmonic 
			for every holomorphic submersion $\Phi$ from a complex 
			manifold where $\Phi^*\omega$ has a global potential.
		\item[(ii)] There is an open cover of $X$ by subsets $U$, with 
			$\omega$-pluripolar subsets $Z \subset U$ and local 
			potentials $\psi$ on $U$, $\psi^{-1}(\{-\infty\}) \subset Z$, 
			such that for every $h \in \A_U$ with $h(\overline \D) \not\subset Z$, 
			the function $t \mapsto \big(H(h(t)) + \psi(h(t))\big)^\dagger$ is dominated by
			an integrable function on $\T$.
		\item[(iii)] If $h \in \A_X$, $h(0) \notin \sing(\omega)$, 
			$t_0 \in \T \setminus h^{-1}(\sing(\omega))$ and $\epsilon > 0$,
			then $t_0$ has a neighbourhood $U$ in $\C$ and there is a local 
			potential $\psi$ in a neighbourhood of $h(U)$
			such that for all sufficiently small arcs $J$ in $\T$ 
			containing $t_0$ there is a holomorphic map
			$F: D_r \times U \to X$, $r>1$, such that $F(0,\cdot) = h|_U$ and
			$$
				\frac{1}{\sigma(J)} \int_J \big( H(F(\cdot,t))+
				\psi(F(0,t)) \big)\ d\sigma(t) \leq (EH+\psi)(h(t_0)) + 
				\epsilon.
			$$
	\end{enumerate}
	\label{th_red}
\end{theorem}
\begin{proof}
	By Proposition \ref{equi} we need to show that 
	$EH \circ h$ is $h^*\omega$-subharmonic for every $h\in \A_X$, 
	$h(\D) \not\subset \sing(\omega)$ and that $EH$ is $\omega$-usc.
	
	The $\omega$-upper semicontinuity of $EH$ follows from Lemma \ref{lemma_redusc} so 
	we turn our attention to the subaverage property.
	We assume $\psi = \psi_1 - \psi_2$ is a local potential of $\omega$ defined on an open 
	set $U$. As with plurisubharmonicity, $\omega$-plurisubharmonicity is a local property
	so it is enough to prove the subaverage property for $h \in \A_U$, $h(0) \notin \sing(\omega)$.
	Our goal is therefore to show that
	\begin{equation}\label{in0}
		EH(h(0)) + \psi(h(0)) \leq \int_\T (EH \circ h + \psi\circ h)^\dagger\, d\sigma.
	\end{equation}
	This is automatically satisfied if $EH(h(0)) = -\infty$, and since $EH$ is $\omega$-usc
	it can only take the value $+\infty$ on $\sing(\omega)$. We may therefore assume $EH(h(0))$
	is finite.
	It is sufficient to show that for every $\epsilon > 0$ and every 
	continuous function $v:U \to \R$
	with $v \geq (EH+\psi)^\dagger$, there exists $g \in \A_X$ such that $g(0) = h(0)$ and
	\begin{equation}\label{in00}
		H(g) + \psi(h(0)) \leq \int_\T v \circ h\, d\sigma + \epsilon.
	\end{equation}
	Then by definition of the envelope, 
	$EH(h(0)) + \psi(h(0)) \leq \int_\T v\circ h \, d\sigma + \epsilon$
	for every $v$ and $\epsilon$, and (\ref{in0}) follows.
	
	Let $r>1$ such that $h$ is holomorphic on $D_r$.
	In the proof of Theorem 1.2 in \cite{LarSig:2003}, L\'arusson and Sigurdsson show that a 
	function satisfying
	the subaverage property for all holomorphic discs in $X$ not lying in a pluripolar 
	set $Z$ is plurisubharmonic not only on $X \setminus Z$ but on $X$.
	We may therefore assume that $h(\overline \D) \not\subset Z$.
	
	Since $h(0) \notin \sing(\omega)$, we have $\psi_1\circ h(0) \neq -\infty$ and 
	$\psi_2\circ h(0) \neq -\infty$. Then, by the subaverage property of the 
	subharmonic functions $\psi_1\circ h$ and $\psi_2\circ h$, the set 
	$h^{-1}(\sing(\omega))$ is of measure zero with respect to the arc length
	measure $\sigma$ on $\T$. The set $h(\T) \setminus \sing(\omega)$ is 
	therefore dense in $h(\T)$ and by a compactness argument along with property \textsl{(iii)} 
	we can find a finite number of closed arcs
	$J_1,\ldots,J_m$ in $\T$, each contained in
	an open disc $U_j$ centered on $\T \setminus \sing(\omega)$ and holomorphic maps 
	$F_j: D_s \times U_j \to X$, $s \in ]1,r[$
	such that $F_j(0,\cdot) = h|_{U_j}$ and,
	using the continuity of $v$, such that
	\begin{equation} \label{in1}
		\underline{\int_{J_j}} \Big( H(F_j(\cdot,t))+\psi (F(0,t)) \Big)\, d\sigma(t) 
		\leq \int_{J_j} v\circ h\, d\sigma + \frac{\epsilon}{4}\sigma(J_j).
	\end{equation}
	We can shrink the discs $U_j$ such that they are relatively compact in $D_r$ and 
	have mutually disjoint closure.
	Furthermore, by the continuity of $v$ we may assume 
	\begin{equation}\label{in2}
		\int_{\T \setminus \cup_j J_j} |v\circ h|\, d\sigma < \frac{\epsilon}{4}
	\end{equation}
	and by condition \textsl{(ii)} we may assume
	\begin{equation}\label{in3}
		\overline {\int_{\T \setminus \cup_j J_j}} H(h(w))+
		\psi(h(w))\, d\sigma(w) < \frac{\epsilon}{4}.
	\end{equation}
		
	Our submersion $\Phi$ will be the projection $\C^2 \times X \to X$. The manifold in $\C^2 \times X$ 
	where $\Phi^*\omega$ has a global potential will be a neighbourhood of the 
	union of the graphs of $h$,
	$$
	K_0 = \{ (w,0,h(w)) ; w \in \overline \D \},
	$$
	and the graphs of the $F_j$'s,
	$$
	K_j = \{ (w,z,F_j(z,w)) ; w \in J_j, z\in \overline \D \}.
	$$
	
	By applying Lemma \ref{lemma_potential} to both
	$\omega_1$ and $\omega_2$ there is neighbourhood $V$ of $K=\cup_{j=0}^m K_j$ with a potentials $\Psi_1$
	of $\Phi^*\omega_1$ and $\Psi_2$ of $\Phi^*\omega_2$. Then $\Psi = \Psi_1 - \Psi_2$ is a potential
	of $\Phi^*\omega$. The $\Phi^*\omega$-plurisubharmonicity of $E\Phi^*H$ given
	by condition \textsl{(i)} ensures 
	\begin{equation}
		E\Phi^*H(\tilde h(0)) + \Phi^*\psi(\tilde h(0)) 
		\leq \int_\T (E\Phi^*H \circ \tilde h + \Phi^*\psi \circ \tilde h)^\dagger\, d\sigma,
		\label{ineq_cond1}
	\end{equation} 
	where $\tilde h$ is the lifting $w\mapsto (w,0,h(w))$ of $h$ to 
	$V \subset \C^2 \times X$.
	
	We know $\Phi^*EH(\tilde h(0)) \leq E\Phi^*H(\tilde h(0))$ and since $\Phi^*EH(\tilde h(0)) = EH(h(0)) \neq -\infty$
	there is a disc $\tilde g \in \A_V$ such that $\tilde g(0) = \tilde h(0)$ and 
	\begin{equation} 
		\Phi^*H(\tilde g) \leq E\Phi^*H(\tilde h(0)) + \frac \epsilon 4.
		\label{ineq_tildeg}
	\end{equation}
	Let $g = \Phi \circ \tilde g$ be the projection of $\tilde g$ to $X$, then $g(0) = h(0)$ and
	$H(g) = \Phi^*H(\tilde g)$. Because the local potential $\Phi^*\psi$ of $\Phi^*\omega$ 
	satisfies $\Phi^*\psi(\tilde h) = \psi(h)$. The inequalities (\ref{ineq_cond1}) and (\ref{ineq_tildeg}) above then
	imply that
	\begin{equation}
		H(g) + \psi(h(0)) \leq \int_\T (E\Phi^*H\circ \tilde h + \psi\circ h)\, d\sigma + \frac \epsilon 4.
		\label{ineq_tildeg2}
	\end{equation}
	For every $j=1,\ldots,m$ and $w \in J_j$ we have
	\begin{equation*}
		E\Phi^*H(\tilde h(w)) \leq \Phi^*H( (w,\cdot,F_j(\cdot,w) ) ) = H(F_j(\cdot,w)),
	\end{equation*}
	because $z\mapsto (w,z,F_j(z,w))$ is a disc in $K$ with center $\tilde h(w)$.

	This means, by (\ref{in1}),
	\begin{equation}
		\int_{J_j} (E\Phi^*H(\tilde h) + \psi \circ h)\, d\sigma \leq \int_{J_j} v\circ h\, d\sigma 
		+ \frac \epsilon 4 \sigma(J_j).
		\label{in4}
	\end{equation}

	But if $w \in \T \setminus \cup_j J_j$ then
	\begin{equation*}
		E\Phi^*H(\tilde h(w)) \leq \Phi^*H(\tilde h(w)) = H(h(w)),
	\end{equation*}
	where $\tilde h(w)$ and $h(w)$ on the right are the constant discs at $\tilde h(w)$ and $h(w)$.
	This means, by (\ref{in3}),
	\begin{equation}
		\int_{\T \setminus \cup_j J_j} (E\Phi^*H(\tilde h) + \psi\circ h)\, d\sigma \leq \frac \epsilon 4
		\label{in5}
	\end{equation}

	Then, first by combining inequality (\ref{ineq_tildeg2}) with (\ref{in4}) and (\ref{in5}), 
	and then by (\ref{in2}), we see that
	\begin{equation*}
		H(g) + \psi(h(0)) \leq 
		\int_{\cup_j J_j} v\circ h 
		+ \frac \epsilon 4 \sigma(\cup_j J_j) + \frac \epsilon 4 + \frac \epsilon 4
		\leq \int_\T v \circ h + \epsilon.
	\end{equation*}
	This shows that the disc $g$ satisfies (\ref{in1}) and we are done.
\end{proof}

\begin{prooftx}{Proof of Theorem \ref{th_main} when $\phi_2=0$}
	Finally, we can prove Theorem \ref{th_main} when $\phi_2=0$ by showing that $H_{\omega,\phi}$ satisfies the
	three condition in Theorem \ref{th_red}.
	
	Condition \textsl{(i)} in \ref{th_red} follows from the proof in Section 4. If $h \in \A_X$ and
	$\psi$ is local potential as in Theorem \ref{th_red}, then condition \textsl{(ii)} follows from the fact
	that $H(h(t)) + \psi(h(t)) = (\phi(h(t)) + \psi_1(h(t))) - \psi_2(h(t))$ is the difference of an usc 
	function and a psh function. The first term is bounded above on $\T$ and the second one is integrable
	since $h(\D) \not\subset \sing(\omega)$.
	
	Let $h \in \A_X$, $\epsilon >0$ and $t_0 \in \T \setminus h^{-1}(\sing(\omega))$ be as in
	condition \textsl{(iii)} and $\psi$ a local potential for $\omega$ in a neighbourhood $V'$ of
	$x=h(t_0)$. Let $\beta > EH_{\omega,\phi}(x) + \psi(x)$ and $\epsilon >0$ such that   
	$EH_{\omega,\phi}(x) + \psi(x) + \epsilon < \beta$. Then there is a $f\in \A_X$ such that
	$f(0)=x$ and $H_{\omega,\phi}(f) + \psi(x) \leq \beta - \epsilon/2$. 
	By Lemma 2.3 in	\cite{LarSig:1998} there is a neighbourhood $V$ of $x$ in $X$, 
	$r>1$ and a holomorphic map $\tilde F:D_r \times V \to X$ such that 
	$\tilde F(\cdot,x)=f$ on $D_r$ and $\tilde F(0,z) = z$ on $V$. 
	Define $U = h^{-1}(V' \cap V)$ and $F:D_r \times U \to X$ by $F(s,t)=\tilde F(s,h(t))$,
	then by (\ref{riesz}),
	\begin{equation}\label{riesz2}
		\big(H_{\omega,\phi}(F(\cdot,t)) + \psi(F(0,t))\big)^\dagger = 
		\int_\T (\phi + \psi)^\dagger \circ F(s,t)\, d\sigma(s).
	\end{equation}
	Since the integrand is usc on $D_r \times U$, then (\ref{riesz2}) is an
	usc function of $t$ on $U$ by Lemma \ref{lemma_composition}. 
	That allows us by shrinking $U$ to assume that
	$$
		\big(H_{\omega,\psi}(F(\cdot,t))+\psi(F(0,t))\big)^\dagger 
		\leq H_{\omega,\phi}(F(\cdot,t_0)) + \psi(F(0,t_0)) + \frac \epsilon 2
	$$
	for $t \in U$. Then by the definition of $f = F(\cdot,t_0)$
	$$
		\big(H_{\omega,\phi}(F(\cdot,t)) + \psi(F(0,t))\big)^\dagger < 
		EH_{\omega,\phi}(x) + \psi(x) + \epsilon, \quad \text{for } t \in U.
	$$
	Condition \textsl{(iii)} is then satisfied for all arcs $J$ in $\T \cap U$.
\end{prooftx}

We now finish the proof of our main theorem by showing how the function $\phi_2$ can be integrated
into $\omega$ and then previous result applied. So, subtracting the function $\phi_2$ from $\phi_1$ can be thought of 
as just shifting the class $\PSH(X,\omega)$ by $-dd^c \phi_2$.
\begin{prooftx}{End of proof of Theorem \ref{th_main}}	
	We define the current $\omega' = \omega - dd^c \phi_2$ and use the bijection, $u' \mapsto u'- \phi_2=u$ between 
	$\PSH(X,\omega')$ and $\PSH(X,\omega)$ from Proposition \ref{prop_biject}. 
	Since the positive part of $\omega$ and $\omega'$ is the same, it is equivalent for 
	$\phi_1$ to be $\omega_1$-usc and $\omega_1'$-usc. Then
	Theorem \ref{th_main} can be applied to $\omega'$ and $\phi_1$, and 
	for every $x \notin \sing(\omega') = \sing(\omega) \subset \phi_2^{-1}(-\infty)$ 
	 we infer
	\begin{eqnarray*}
		&& \sup\{u(x) ; u \in \PSH(X,\omega), u \leq \phi_1 - \phi_2 \}\\
		&=& \sup\{u'(x) - \phi_2(x) ; u' \in \PSH(X,\omega'), u'-\phi_2 \leq \phi_1 - \phi_2 \}\\
		&=& \sup\{u'(x)  ; u' \in \PSH(X,\omega'), u' \leq \phi_1 \} - \phi_2(x) \\
		&=& \inf\{ -R_{f^*\omega'}(0) + \int_\T \phi_1 \circ f\, d\sigma ;  
		f \in \A_X, f(0) = x \} - \phi_2(x)\\
		&=& \inf\{ -R_{f^*\omega}(0) +R_{f^*dd^c \phi_2}(0) - \phi_2(x) + \int_\T \phi_1 \circ f\, d\sigma ;  
		f \in \A_X, f(0) = x \}\\
		&=& \inf\{ -R_{f^*\omega}(0) + \int_\T (\phi_1-\phi_2) \circ f\, d\sigma ;  f \in \A_X, f(0) = x \}.
        \end{eqnarray*}
	The last equality follows from the Riesz representation (\ref{riesz}) applied to the
	psh function $\phi_2$, which gives
	$\phi_2(x) = R_{f^*dd^c \phi_2}(0) + \int_\T \phi_2 \circ f\, d\sigma$. We also used the fact that
	$R_{f^*\omega}$ is linear in $\omega$.

	To finish the proof we need to show that the equality  
	\begin{multline} \sup\{u(x) ; u \in \PSH(X,\omega), u \leq \phi_1 - \phi_2 \}\\
		= \inf\{ -R_{f^*\omega}(0) + \int_\T (\phi_1-\phi_2) \circ f\, d\sigma ;  f \in \A_X, f(0) = x \},
		\label{final_eq}
	\end{multline}
	holds also on $\phi_2^{-1}(-\infty) \setminus \sing(\omega)$.

	The right hand side of (\ref{final_eq}) is $\omega$-usc by Lemma 
	\ref{lemma_redusc}, and it is equal to 
	the function $EH_{\omega',\phi_1}-\phi_2$ on $X\setminus \sing(\omega')$. 
	Now assume $\psi$ is a local
	potential of $\omega$, then $-\phi_2 + \psi$ is a local potential for
	$\omega'$. The functions $(EH_{\omega',\phi_1}-\phi_2+\psi)^\dagger$ and
	$(EH_{\omega,\phi} + \psi)^\dagger$ are then two usc functions which are
	equal almost everywhere, thus the same. 
	Furthermore, since $EH_{\omega',\phi_1}$ is $\omega'$-psh we see that 
	$EH_{\omega,\phi}$ is $\omega$-psh. This 
	shows that $EH_{\omega,\phi}$ is in the family 
	$\{ u\in \PSH(X,\omega), u \leq \phi\}$, and since 
	$\sup\{u \in \PSH(X,\omega); u\leq \phi\} \leq EH_{\omega,\phi}$ 
	by (\ref{fund_ineq}) we have an equality
	not only on $X\setminus \sing(\omega')$ but on $X\setminus \sing(\omega)$,
	i.e.~(\ref{final_eq}) holds on $X\setminus \sing(\omega)$.
\end{prooftx}

\bibliographystyle{siam}
\bibliography{bibref}

\bigskip
{\small
Science Institute, University of Iceland, Dunhaga 3, IS-107
Reykjavik, Iceland

E-mail: bsm@hi.is
}

\end{document}